\documentclass[11pt]{amsart}
\usepackage{amssymb,amsmath,amstext}

\theoremstyle{plain}

\newtheorem{theorem}{Theorem}[section]
\newtheorem{assertion}[theorem]{Assertion}

\newtheorem{lemma}[theorem]{Lemma}

\newtheorem{corollary}[theorem]{Corollary}

\newtheorem{conjecture}[theorem]{Conjecture}

\theoremstyle{definition}
\newtheorem{definition}[theorem]{Definition}

\theoremstyle{remark}
\newtheorem{example}[theorem]{Example}

\newcommand{\F}{\mathcal{F}}
\newcommand{\I}{\mathcal{I}}

\newcommand{\M}{\mathcal{M}}

\newcommand{\T}{\mathcal{T}}

\newcommand{\e}{\varepsilon}

\begin{document}

\title[Fractional covers of $d$-intervals]{Fractional covers and matchings in families of weighted $d$-intervals}

\author{Ron Aharoni}\thanks{The research of the first  author was
supported by  BSF grant no.
2006099, by an ISF grant and by the Discount Bank
Chair at the Technion.}
\address{Department of Mathematics,
Technion\\
Haifa, Israel} \email{raharoni@gmail.com}
\author{Tom\'{a}\v{s} Kaiser}\thanks{The second author was supported by project 14-19503S of the
Czech Science Foundation.}
\address{Department of Mathematics, Institute for Theoretical Computer Science
(CE-ITI), and European Centre of Excellence NTIS (New Technologies for
the Information Society), University of West
  Bohemia\\Pilsen, Czech Republic}\email{kaisert@kma.zcu.cz}
\author{Shira Zerbib}
\address{Department of Mathematics,
Technion\\
Haifa, Israel} \email{shira.zerbib@gmail.com}


\begin{abstract}
  A $d$-{\em interval} is a union of at most $d$ disjoint closed
  intervals on a fixed line. Tardos \cite{tardos} and the second
  author \cite{kaiser} used topological tools to bound the transversal
  number $\tau$ of a family $H$ of $d$-intervals in terms of $d$ and
  the matching number $\nu$ of $H$. We investigate the weighted and fractional
  versions of this problem and prove upper bounds that are tight up to
   constant factors. We apply both the topological method and an
  approach of Alon \cite{alon}. For the use of the latter, we prove a
  weighted version of Tur\'{a}n's theorem. We also provide a proof of
  the upper bound of \cite{kaiser} that is more direct than the
  original proof.
  \\ \ \\

  \end{abstract}

\maketitle


\section{Introduction}
\label{sec.int}

A {\em $d$-interval} is the union of at most $d$ disjoint closed
intervals on a fixed line. The $j$-th component of a $d$-interval $h$,
counted from the left to right, will be denoted by $h^j$.

We call a $d$-interval $h$ \emph{separated} if its intersection with
each interval $(i,i+1)$, where $0\leq i < d$, is either empty or
coincides with one component of $h$. For our purposes, we can
equivalently picture a separated $d$-interval as the union of $d$
possibly empty intervals, one on each of $d$ fixed parallel lines
(which is the definition used in~\cite{tardos2}). We shall also
consider {\em discrete} $d$-intervals where the lines are replaced by
finite linearly ordered sets. All results and conjectures below are
easily seen to be equivalent to their discrete versions, and sometimes
we shall switch to the discrete versions without further argument. We
shall assume that our hypergraphs are finite.

We will be interested in properties of hypergraphs whose vertex sets
are the points of the lines considered above, and whose edges are
$d$-intervals. We call them \emph{hypergraphs of $d$-intervals}.

A {\em matching} in a hypergraph $H=(V,E)$ with vertex set $V$ and
edge set $E$ is a set of disjoint edges. A {\em cover} is a subset of
$V$ meeting all edges. The \emph{matching number} $\nu(H)$ is the maximal size of a matching, and
the \emph{covering number}, or  \emph{transversal number} $\tau(H)$ is the minimal size of a
cover. The fractional relaxations of $\nu$ and $\tau$ are
denoted, as usual, by $\nu^*$ and $\tau^*$, respectively. By Linear Programming duality, $\nu^*=\tau^*$.

An old result of Gallai is that if $H$ is a hypergraph of ($1$-)intervals, then $\tau(H)=\nu(H)$. In~\cite{gl1, gl2} the authors raised the problem of the best bound on the ratio $\frac{\tau}{\nu}$
in hypergraphs of $2$-intervals, with the natural conjecture being that the answer is $2$.
This was proved by Tardos \cite{tardos},
using algebraic topology. A simpler proof was then found by the
second author \cite{kaiser}, who also extended the theorem to
hypergraphs consisting of $d$-intervals:

\begin{theorem}\label{t:kaiser}
  For a hypergraph $H$ of
  $d$-intervals,
  \begin{equation*}
    \frac{\tau(H)}{\nu(H)} \leq d^2-d+1.
  \end{equation*}
  In the separated case, the upper bound improves to $d^2-d$.
\end{theorem}
Matou\v{s}ek \cite{matousek} showed that this bound is not far from
the truth: there are examples of intersecting hypergraphs of
$d$-intervals in which $\tau=\frac\tau\nu = \Omega(\frac{d^2}{\log
  d})$. The proof of Theorem~\ref{t:kaiser}
in \cite{kaiser} relies on reducing the problem to a discrete problem
on $d$-uniform hypergraphs and the fact that in such hypergraphs
$\tau^* \le d\nu$.

Another approach was taken by Alon \cite{alon} who proved a slightly weaker
upper bound for the non-separated case, namely $\frac{\tau}{\nu} \leq
2d^2$. Like in~\cite{kaiser}, there is a central role in this proof to fractional versions.
 Alon breaks his bound into the following two results:

\begin{theorem}\label{nutau*}
$\tau^* \le 2d\nu$.
\end{theorem}

\begin{theorem}\label{nu*tau}
$\tau \le d\tau^*$.
\end{theorem}

Theorem \ref{nu*tau} follows by a relatively simple sampling
argument. (As will be shown below, it also follows from the arguments
in \cite{kaiser}.)
The bound in Theorem \ref{nutau*} is proved using
Tur\'{a}n's theorem.

As noted, both approaches use fractional parameters.  This makes
finding the right relations between them and the integral parameters
of interest. In particular, the following conjecture is appealing:

\begin{conjecture}\label{conj:tau*nu}
  In a hypergraph of separated $d$-intervals, $\tau^* \le d\nu$.
\end{conjecture}

The conjecture is not true in the non-separated case, as shown by an
intersecting family of 2-intervals with $\tau^*\geq\tfrac{11}4$
constructed by Gy\'{a}rf\'{a}s~\cite[p.~45]{Gya}. An open question
in~\cite{Gya} is whether $\tau^*<3$ for intersecting families of
(non-separated) 2-intervals.


Besides the fractional parameters, we shall study also the weighted
versions of the matching and covering numbers, as defined below. For a real valued
function $f$ on a set $T$, we let $f[T]$ denote the sum $\sum_{t\in
T}f(t)$.

\begin{definition}
Given a hypergraph $H$, a \emph{weight system on $E(H)$} is a function
 $w:~ E(H) \to \mathbb{N}$. A function $g: V(H) \to \mathbb{N}$ is called a {\em $w$-cover} if $\sum_{v \in e}g(v) \ge w(e)$ for all $e \in E(H)$.
 Let $\nu_w(H)$ be the maximum of $w[F]$ over all matchings $F$ in $H$, and let
$\tau_w(H)$ be the minimum of $g[V]$ over all $w$-covers $g$.
\end{definition}

It is likely that the bounds (conjectured as well as proved) that are true in the non-weighted case
are true also in the weighted case.

\begin{conjecture}\label{tauw*nuw}
  Let $H$ be a hypergraph of separated $d$-intervals and $w$ a weight
  system on $E(H)$. Then $\tau^*_w(H) \le d\nu_w(H)$.
\end{conjecture}

In  Section \ref{sec.turan} we shall prove a weighted and directed version of Tur\'an's theorem, and in Section \ref{weightedalon} we shall use it to prove:
\begin{theorem}\label{tauw*}
If $H$ is a hypergraph of  $d$-intervals
  and $w$ a weight system on $E(H)$, then $\tau^*_w(H) \le 2d \nu_w(H)$.
\end{theorem}

A straightforward generalization of the proof of Theorem \ref{nu*tau} will yield:
\begin{theorem}\label{wtautau*} If $H$ is a hypergraph of  $d$-intervals
  and $w$ a weight system on $E(H)$, then
$\tau_w(H) \le d\tau^*_w(H)$.
\end{theorem}

Combining the two results, we have:
\begin{corollary}\label{2d2}
If $H$ is a hypergraph of  $d$-intervals
  and $w$ a weight system on $E(H)$,  then $\tau_w(H) \le 2d^2 \nu_w(H)$.
\end{corollary}

By Theorem \ref{wtautau*},  Conjecture \ref{tauw*nuw}, if true,
would imply the separated case of the following:

\begin{conjecture}\label{d2w} In a hypergraph of  $d$-intervals
  $\tau_w \le d^2\nu_w$.
\end{conjecture}

This conjecture is true for intersecting hypergraphs, since by
Theorem~\ref{t:kaiser}, there exists a cover of size at most
$d^2-d+1$, and putting weight $\max_{h \in H}w(h)=\nu_w$ on each of
its points constitutes a weighted cover.

In Section \ref{topfrac} we shall show that the case of the weight
of each $d$-interval  being its length has a special stature, namely
that in order to prove an upper bound of $\alpha d^2$ on
$\frac{\tau_w}{\nu_w}$, it is enough to prove that in every
hypergraph $H$ of $d$-intervals
$$\frac{\tau^*_{\ell}}{\nu_{\ell}} \le \alpha d,$$ where $\ell(h)$ is the length of $h\in H$.
The proof uses the KKMS  (Knaster-Kuratowski-Mazurkewitz-Shapley) theorem.
 In Section \ref{sec.tardos-kaiser} we shall use the same theorem, and a generalization of it due to Komiya,  to provide a
simplified proof of Theorem \ref{t:kaiser}. In Section
\ref{toptau*nu} we shall use the KKMS theorem to prove an upper
bound of 4d on $\frac{\tau^*}{\nu}$.
This result is weaker than Theorem  \ref{nutau*}, but we think it is still worth presenting, since an improved
topological approach may well be the right tool for
finding the right bound.

\section{Examples for  sharpness}

If true, then Conjectures \ref{conj:tau*nu} and \ref{tauw*nuw} are
sharp for all $d$. In the next example $\nu=1$ and $\tau^*=d$:

\begin{example}\label{example}
It is known that the edge set of $K_{2d}$ is decomposable into $d$
Hamiltonian paths $P_1, \ldots ,P_d$ (see \cite{stanton}). Let
$V(K_{2d})=[2d]$. Each path $P_j$ can be viewed as a permutation
$\pi_j$ on $[2d]$. Let $H=\{e_1, \ldots ,e_{2d}\}$, where $e_i$ is the
separated interval whose $j$-th component is
$e^j_i=[\pi^{-1}_j(i), \pi^{-1}_j(i)+1]$.  For every $k,\ell \in [2d]$ the
edge $(k,\ell)$ of $K_{2d}$ belongs to some path $P_j$, namely
$k=\pi_j(i),~\ell=\pi_j(i+1)$ or $\ell=\pi_j(i),~k=\pi_j(i+1)$ for
some $i$. This implies that $e_k$ and $e_\ell$ meet in the $j$-th
line. Thus $\nu(H)=1$. Putting weight $\frac{1}{2}$ on each $e_i$
yields a fractional matching of size $d$. The set of points $\{2k\mid 1\le k\le d\}$ on the first line constitutes a cover of $H$ of size $d$. This proves that $\tau(H)=\tau^*(H)=d$.

\end{example}

The bound given in Theorem
\ref{nu*tau} is  also  sharp, at least asymptotically, even for
separated $d$-intervals. This is shown by the following example.
 Let $n$ be an integer,
let $[0,1]^{\cup d}$ denote the union of $d$ disjoint copies of the
unit interval, and take $H$ to
be the set of $d$-intervals $e \subseteq [0,1]^{\cup d}$ satisfying $|e|
> \frac{1}{n}$, where $|f|$ denotes the total length of a $d$-interval $f$, namely the sum of the lengths of its components.

\begin{assertion}
$\tau(H) \ge nd^2-d$.
\end{assertion}
\begin{proof}
Let $C$ be a cover for $H$. Let $C_i$ be obtained from the
intersection of $C$ with the $i$-th line by the addition of $0$ and
$1$, and let $\ell_i$ be the longest distance between two
consecutive points of $C_i$. Then, clearly, $\sum_{i \le d}\ell_i\le
\frac{1}{n}$, and the intersection of $C$ with the $i$-th line
contains at least $\frac{1}{\ell_i}-1$ points from $C$. The latter
means that $|C| \ge \sum_{i \le d} \frac{1}{\ell_i} -d$. By the
harmonic-arithmetic average inequality, $\sum_{i \le
d}\frac{1}{\ell_i} \ge \frac{d^2}{\sum_{i \le d}\ell_i}\ge nd^2$,
and hence $|C| \ge nd^2-d$.
\end{proof}

\begin{assertion} $\nu^*(H) \le nd$. \end{assertion}
\begin{proof}
Let $\alpha$ be a fractional matching in $H$. Denoting the value of
$\alpha$ on an edge $e$ by $\alpha_e$, and the characteristic function
of $e$ by $I_e$, we have
$$d=|[0,1]^{\cup d}| \ge \int \sum \alpha_e I_e = \sum \alpha_e \int
I_e = \sum \alpha_e |e|,$$ where the integration is over $[0,1]^{\cup
  d}$. Since $|e| > \frac{1}{n}$ for all $e\in H$ it follows that
$\sum \alpha_e < nd$. Alternatively, the constant function $n$
constitutes a `continuous fractional cover' of $H$ of size $nd$, and
it can be approximated by discrete fractional covers as well as we
please.
\end{proof}

By the above, for every $n$ there exists a $d$-interval hypergraph
with $\frac{\tau}{\nu^*} \ge d-\frac{1}{n}$.


\section{A weighted version of Tur\'{a}n's theorem}
\label{sec.turan}

In this section we prove a weighted version of Tur\'{a}n's theorem.
Let $G=(V,E)$ be a graph.  Recalling notation defined above, for a set $X$ of vertices  $w[X]$ is  $\sum_{x
  \in X}w(x)$.  For a set $F$ of edges let $\tilde{w}[F]=\sum_{uv \in
  F}(w(u)+w(v))$.  If $G$ is directed, given disjoint subsets $A$ and $B$ of $V$, we write
$E(A,B)=\{xy \mid x \in A, y \in B\}$, where $xy$ is an edge directed from $x$ to $y$. If $v \in V\setminus A$ we
write $E(v,A)$ and $E(A,v)$ for $E(\{v\},A)$ and $E(A,\{v\})$, respectively.
The set of neighbors of a vertex $v$ is
$N(v)=\{u \mid uv\in E ~or~ vu \in E\}$. We let $\alpha_w(G)$ denote the maximal sum of weights on an
independent set in $V$.

\begin{theorem}\label{wturan}
Let $G=(V,E)$ be a graph and let $w:~V \to \mathbb{N}^+$ be a weight
function on its vertices. Let $W=w[V]$ and $K=\alpha_w(G)$.
 Then
\begin{equation}\tilde{w}[E] \ge \frac{W^2}K-W.
\label{wt}
\end{equation}
\end{theorem}

Note that if all  weights are  1, this is just Tur\'{a}n's
theorem. We will actually need a result about directed graphs that
implies Theorem~\ref{wturan} as an easy corollary, upon replacing every edge in $G$ by two oppositely directed edges:

\begin{theorem}\label{wtrefined}
  Let $D$ be a directed graph in which every pair of adjacent vertices
  is connected by at least two directed edges, not necessarily in the
  same direction. Let $w$ be a weight function on $V=V(D)$, and write
  $W = w[V]$ and $K=\alpha_w(D)$. Then
  $$\sum_{\ xy \in E(D)}w(x) \ge \frac{W^2}K - W.$$
\end{theorem}

\begin{proof}
Choose an independent set $A$ of total weight $K$.
Write $B=V\setminus A$. By the induction hypothesis
\begin{equation}\label{inductionb}
  \sum_{xy \in E(D[B])} w(x)\ge
 \frac{w[B]^2}K - w[B]. \end{equation}

   Let $v$ be a vertex not in $A$. By the maximality property of $A$,
  we have
  \begin{equation}\label{eq:max}
     \sum\{w(u) \mid  u\in A \cap N(v)\} \ge w(v),
    \end{equation}
  for otherwise the sum of weights in $A \setminus N(v) + v$ would be
  larger than in $A$. In particular, note that $|A \cap N(v)| \ge 1$.

  We claim that

\begin{equation}\label{2wv} \sum_{uv \in E(A,v)}w(u)+\sum_{vu \in E(v,A)} w(v) \ge
2w(v). \end{equation}
Indeed, if all $A-v$ edges are directed from $A$ to $v$, this follows from (\ref{eq:max}). If there exist
two (possibly parallel) edges directed from $v$ to $A$ then
\eqref{2wv} is obvious. So, there remains the case that there is
precisely one $A-v$ edge, say $va$, directed from $v$ to $A$. Then
$$\sum_{uv \in E(A,v)}w(u)+\sum_{vu \in E(v,A)} w(v) = 2\sum\{w(u) \mid  u\in A \cap N(v)\}-w(a)+w(v).$$ Therefore if $w(a)\le w(v)$ then \eqref{2wv} is proved by \eqref{eq:max}. Else we have $$\sum_{uv \in E(A,v)}w(u)+\sum_{vu \in E(v,A)} w(v) \ge w(v) + w(a) > 2w(v),$$ as claimed.

     Hence

     \begin{equation}\label{2wb} \sum_{uv \in E(A,B)}w(u)+\sum_{vu \in E(B,A)} w(v) \ge 2w[B]. \end{equation}

  Using \eqref{inductionb} and \eqref{2wb}, we get:
  \begin{equation}
    \label{conclusion}
    \sum_{\ xy \in E(D)}w(x) \ge  \frac{w[B]^2}K - w[B] + 2w[B]=\frac{w[B]^2}{K} +
    w[B].
  \end{equation}

  Since $W=w[A]+w[B]$ and $w[A] = K$, we have
  \begin{align*}
    \frac{W^2}{K}-W &= \frac{w[A]^2+2w[A]w[B]+w[B]^2}{K}-w[A]-w[B]\\
    &=\frac{w[B]^2}{K} + w[B].
  \end{align*}
  Together with \eqref{conclusion} this proves the desired inequality.

\end{proof}


\section{Weighted matchings}\label{weightedalon}
In this section we prove Theorems \ref{tauw*} and \ref{wtautau*}. But
before that, here is some
 motivation to the study of the weighted case:  a
(well known) connection to edge-colorings. The {\em edge chromatic
  number} $\chi_e(H)$ of a hypergraph $H$ is the minimal number of
matchings needed to cover all edges of the hypergraph.  A {\em
  fractional edge coloring} is a non-negative function $f$ on the set
$\M$ of matchings in $H$, satisfying the condition that $\sum_{e \in M
  \in \M} f(M)\ge 1$ for all $e \in E(H)$.  The fractional edge
chromatic number $\chi^*_e(H)$ is the minimum, over all fractional
edge colorings $f$ of $H$, of $f[\M]$. (It is
easy to see that this minimum exists.)

\begin{lemma}\label{fracedgechromaticnuw}
For any  hypergraph $H=(V,E)$, if $\tau^*_w(H) \le \alpha \nu_w(H)$
for every weight system $w$ on $E(H)$ then $\chi^*_e(H) \le
\alpha\Delta(H)$.

\end{lemma}

\begin{proof}
Write $q$ for $\chi^*_e(H)$. By LP duality, $q=\max w[E]$, where the
maximum is over all weight functions $w$ on $E$ satisfying the
condition that $w[M]\le 1$ for every matching $M$ in $H$. Taking $w$
for which this maximum is attained, we have $\nu_w(H) = 1$, and
hence by the assumption  of the lemma, $\tau^*_w(H) \le \alpha$. But
since every vertex participates in at most $\Delta(H)$ edges in a
fractional $w$-covering of the edges, $\tau^*_w \ge
\frac{w[E]}{\Delta(H)}=\frac{q}{\Delta(H)}$. Combining these
inequalities yields  $q \le \alpha \Delta(H)$, as desired.
\end{proof}

In \cite{gl2} the conjecture was raised that in a hypergraph $H$ of  $2$-intervals $\chi_e(H)$ does not exceed
$2$ times the maximal size of an intersecting subhypergraph. The following conjecture strengthens this, and generalizes it to all $d$:

\begin{conjecture}\label{conj:edgecoloring}
If $H$ is a hypergraph of  $d$-intervals, and the maximum
degree of a point on any line is $\Delta$, then $\chi_e(H) \le
d\Delta$.
\end{conjecture}

If true, then this is sharp by Example \ref{example}. There $\Delta=2$, and since $\nu=1$ we have $\chi_e=|E|=2d$.

As noted already in~\cite{gl2}, relaxing the problem by a factor of $2$
brings Conjecture \ref{conj:edgecoloring} within easy reach:

\begin{theorem} Under the assumptions of
  Conjecture~\ref{conj:edgecoloring}, $\chi_e(H) \le 2d(\Delta-1)$.
\end{theorem}

\begin{proof}
  Let $D$ be a digraph whose vertex set is $E(H)$, and in which $e$
  sends an arrow to $f$ if an endpoint of some interval component of
  $e$ belongs to $f$. Then the outdegree of every vertex of $D$ (=edge
  of $H$) is at most $2d(\Delta-1)$, and hence the number of directed
  edges is at most $2d(\Delta-1)|V|$. As is easy to realize, for every
  pair of vertices of $D$, if there is an edge between them then there
  are two, hence the number of edges in the underlying undirected
  graph $G$ is at most $d(\Delta-1)|V|$, so the average degree in $G$
  is at most $2d(\Delta-1)$. Since the same is true for any induced
  subgraph of $G$, the theorem follows by successively removing
  vertices of minimum degree and coloring them greedily in the reverse
  order.
\end{proof}

By Lemma \ref{fracedgechromaticnuw}, Conjecture \ref{tauw*nuw} would
imply the fractional version of Conjecture \ref{conj:edgecoloring},
namely $\chi^*_e(H) \le d\Delta(H)$, for separated $d$-intervals.

\bigskip

The proof of Theorem \ref{wtautau*} is almost identical to the proof in the non-weighted case:

{\em Proof of Theorem \ref{wtautau*}}.~
Let $g$ be a rational valued fractional $w$-cover of $H$ of minimal size, and let $|g|=\frac{p}{q}$. By duplicating points we may assume that $g$ has value $\frac{1}{q}$ on  each point belonging to a
set $P$ of $p$ points, and that $d | q$. Write
$m=\frac{q}{d}$.
 Let $Q$ be a set obtained by taking every $m$-th point in $P$, in the left to right order on the line. Every  $e\in H$ satisfies $|P \cap e| \ge qw(e)$,
  and hence has a component containing $\frac{qw(e)}{d}$ points from $P$. This component contains then at least $\frac{qw(e)}{dm}=w(e)$ points from $Q$. Thus $Q$ is a $w$-cover, and its size is
  $\frac{p}{m}=\frac{pd}{q}=d\tau_w^*$.
\hfill{$\square$}
\\ \ \\

\emph{Proof of Theorem \ref{tauw*}}. ~~
Write $K=\nu_w(H)$.  We need to show that $\nu_w^*(H)\le 2dK$, meaning that for every fractional matching $f$ we have
$\sum_{e \in H}w(e)f(e) \le 2dK$.
Multiplying by a common denominator, removing edges on which $f=0$
and duplicating edges if necessary, we can assume that $f(e)=\frac{1}{q}$ for all $e
\in H$, for some natural number $q$.  We write $W$ for $w[H]$. In this terminology, we have to show that

\begin{equation}\label{2dk}
\frac{W}{q} \le 2dK.
\end{equation}

 Let $D$ be a digraph obtained from the line graph  $L=L(H)$ by
 duplicating each edge in $E(L)$, and
 directing the two copies of each edge according to the following rule. Suppose that
 two edges $e_1$ and $e_2$ in $H$ meet, namely $e_1e_2 \in E(L)$. Choose
an interval component $c_1$ of $e_1$ that meets an interval
component $c_2$ of $e_2$. It is now possible to choose two pairs
$(x,c_i)$, where $x$ is an endpoint of $c_{3-i}$, and belongs to
$c_i$. For each choice of such a pair direct a copy of the edge
$e_1e_2$ from $e_{3-i}$ to $e_i$ (namely, from the piercing edge to
the pierced one).

Let $\tilde{D}$ be $D$ with all loops $ee$ added. By Theorem
\ref{wtrefined}, $\sum_{e \in H}w(e) deg_D^+(e) \ge W(W-K)/K$ (here
$deg^+$ denotes the outdegree). Dividing both sides by $W$ gives
that the weighted average of $deg_D^+(e)$ (weighted by $w(e)$) is at
least $\frac{W}{K}-1$, and hence the weighted average of
$deg_{\tilde{D}}^+(e)$ is $\frac{W}{K}$. Hence there is an edge
 $a \in H$ with $deg_D^+(a) \ge \frac{W}{K}-1$, and
considering $a$ as adjacent also to itself, its outdegree in
$\tilde{D}$ is at least $\frac{W}{K}$. That is, it pierces by its
endpoints at least this many edges. Since $a$ has at most $2d$ endpoints, it
has an endpoint $x$ meeting at least $\frac{W}{2dK}$ edges. Since $\sum_{e
\in H, ~x \in e}f(e) \le 1$, it follows that $W/(2dKq) \le 1$, proving
\eqref{2dk}. \hfill{$\square$}


\section{KKMS and the special role of lengths}\label{topfrac}

A  natural weight on a $d$-interval is its total length. In this section we  show
that in some sense, to be specified below, this is the  general case.
 The tool showing this is the KKMS theorem (to be given below).

For a hypergraph $H$ define a weight function $\ell=|h|$ for every
$h\in H$. Call \ref{tauw*nuw}L the special case of Conjecture
\ref{tauw*nuw} in which $w=\ell$. Call a hypergraph $H=(V,E)$ {\em
balanced} if it has a perfect fractional matching, namely a weight
system on $E$ such that for each vertex $v$, the weights of the
edges containing $v$ sum up to $1$.

\begin{lemma}\label{discreteweaker}
If a hypergraph $H$ on a vertex set $[k]$ is balanced and
$\tau^*_{\ell}(H) \le \alpha \nu_{\ell}(H)$ then $H$ contains a
matching $M$ such that $\sum_{m \in M}|m| \ge \frac{k}{\alpha}$.
\end{lemma}

\begin{proof}
Let $f$ be a perfect fractional matching on $H$, then $$\sum_{h\in
H} f(h)|h| = k.$$ Therefore, $$\sum_{h\in H} f(h)\ell(h) = k,$$
which implies that $\tau^*_{\ell}=\nu^*_{\ell}(H)\ge k$. Thus we
have $\nu_{\ell}(H) \ge \frac{k}{\alpha}$, which proves the lemma.
\end{proof}

Therefore, by Theorem \ref{tauw*} we have the following:

\begin{corollary}\label{obs:totalhalf}
    In a balanced discrete hypergraph of $d$-intervals on
    $[k]$ there exists a matching of total size $\frac{k}{2d}$. If the $d$-intervals are separated and
    Conjecture \ref{tauw*nuw}L is true, then there exists a matching of
    total size $\frac{k}{d}$.
\end{corollary}

 In the continuous case, for a $d$-intervals hypergraph to be balanced it has to consist of
 intervals that are half closed (say at the left)  and half open (at the right). Assuming this, here is
 an appealing special case of Conjecture \ref{tauw*nuw}L: A balanced hypergraph of separated $d$-intervals in
 $[0,1]^{\cup d}$ has a matching of total length $1$. By the first part of the observation, there exists in such a
 hypergraph a matching of total length $\frac{1}{2}$.

\bigskip

\begin{theorem} \label{reduction}
Let $\alpha$ be such that for every hypergraph of $d$-intervals
$\tau^*_{\ell} \le  \alpha d\nu_{\ell}$. Then for every hypergraph
of $d$-intervals and for every weight function $w$, $\tau_{w} \le
\alpha d^2\nu_{w}$.
\end{theorem}

The simplex $\Delta_k$ is the set of all points $\vec{x}=(x_1,x_2,
\ldots ,x_{k+1})\in \mathbb{R}^{k+1}_+$ satisfying $\sum_{i \le
k+1}x_i=1$. For $S\subseteq [k+1]$ let $F(S)$ be the face of
$\Delta_k$ consisting of the points $\vec{x}$ satisfying $x_i=0$ for
all $i \not \in S$.

\begin{theorem}[KKMS]\label{kkms} Suppose that with every subset $T$ of $[k+1]$ there is associated a
subset $B_T$ of $\Delta_k$, so that all $B_T$ are closed or all of
them are open, and $F(R) \subseteq \bigcup_{T \subseteq R}B_T$ for
every $R \subseteq [k+1]$. Then there exists a balanced set $\T$ of
subsets of $[k+1]$, satisfying $\bigcap_{T \in \T} B_T\neq
\emptyset$.
\end{theorem}

This is a generalization, by Shapley \cite{shapley}, of the famous KKM (Knaster-Kuratowski-Mazurkiewicz) theorem, which is the
case where
the only nonempty $B_T$'s are $B_{\{i\}}$ for singletons $\{i\}
\subset [k+1]$.

\bigskip

{\em Proof of Theorem \ref{reduction}}.~~   We may  assume that all edges of $H$ are contained in $(0,1)$. Assume that $\tau_w(H)>k$. The theorem will
follow if we show that $\nu_w(H) > \frac{k}{\alpha d^2}$. The assumption that $\tau_w(H)>k$ implies that for
every point $\vec{x}=(x_1, x_2, \ldots ,x_{k+1})$ in $P =
\Delta_{k}$, the points
$p_{\vec{x}}(m)=\sum_{i \le m}x_i$, $1 \le m \le k$ do not
constitute a weighted cover for $H$. Note that the points
$p_{\vec{x}}(m)$ are taken with multiplicity, namely if such a point
repeats $q$ times, it is given weight $q$. The above can
be stated as:

\begin{assertion}\label{uncoverededges1}
For each $\vec{x} \in \Delta_k$ there exists $h \in H$ that contains
fewer than $w(h)$ points $p_{\vec{x}}(m)$ (counted with multiplicity).
\end{assertion}

Given $\vec{x}\in P$, let $p_{\vec{x}}(0)=0$ and
$p_{\vec{x}}(k+1)=1$. For $T \subseteq [k+1]$, define $L(T,\vec{x})=
\bigcup_{t \in T}[p_{\vec{x}}({t-1}),p_{\vec{x}}(t)]$.
Let $c(T,\vec{x})$ be the number of connected components of $L(T,\vec{x})$.
For $X \subseteq [0,1]$ let $np(X,\vec{x})$ be the number of points $p_{\vec{x}}(m)$ (counted with multiplicity) in $X$. Note that $np(int(L(T,\vec{x})),\vec{x}) \ge |T| - c(T,\vec{x})$  ($int(X)$ is the interior of $X$).

We want now
to define sets $B_T$, towards an application of the KKMS theorem. A
definition that almost works is this: $\vec{x} \in B_T$ if there exists an edge $h\in H$ that is contained in
$int(L(T,\vec{x}))$ and $w(h)\ge np(int(L(T,\vec{x})),\vec{x}) +1$.
The problem
is that with this definition $B_T$ is not necessarily open (and
obviously also not necessarily closed). For example, let
$T=\{1,3\}$, and let $\vec{x}$ be such that $x_2=0$. It is possible
that there exists $h \in H$ properly contained in $L(T,\vec{x})) =
[x'_0, x'_1] \cup [x'_1+x'_2, x'_1+x'_2+x'_3]= [x_0,x_1+x_2+x_3]$,
so $\vec{x} \in B_T$ by this definition, but for points $\vec{x}'$
arbitrarily close to $\vec{x}$, in which $x'_2>0$, there is no $h
\in H$ contained in $L(T,\vec{x}')=[x'_0, x'_1] \cup [x'_1+x'_2,
x'_1+x'_2+x'_3]$, and so $\vec{x}' \notin B_T$. For this reason, the
definition of $B_T$ will be a bit more involved.

Let $A_{T}(>\e)$ (resp. $A_{T}(\ge
\e)$) be the set of those points $\vec{x}$ for which the following hold:
\begin{itemize}
\item $c(T,\vec{x}) \le d$,
\item there exists an
edge  $h \in H$ contained in $int(L(T,\vec{x}))$, satisfying $w(h)\ge np(int(L(T,\vec{x})),\vec{x}) +1$, and
\item $dist(h,
\partial(L(T,\vec{x}))
> \e$ (resp. $dist(h,
\partial(L(T,\vec{x})) \ge \e$.
\end{itemize}
(Here $\partial(X)$ denotes the boundary of $X$, and $dist$ stands for ``distance''.)

\begin{assertion}
 $F(S) \subseteq \bigcup_{\e>0}
\bigcup_{T \subseteq S}A_{T}(>\e)$ for every subset $S$ of $[k+1]$.
\end{assertion}
\begin{proof}
If $\vec{x} \in F(S)$ then by Assertion \ref{uncoverededges1} there
exists $h \in H$ that is covered by fewer than $w(h)$ points
$p_{\vec{x}}(m)$. Since $h$ has at most $d$ interval components, this means that there exists some $R \subseteq [k+1]$ such that $h \subset int(L(R,\vec{x})$, $w(h) > np(int(L(R,\vec{x})),\vec{x})$, and $c(R,\vec{x}) \le d$. Let
$T = R \cap S$. Since $x_i=0$ for $i \notin S$ we have $int(L(R,\vec{x}))=int(L(T,\vec{x}))$. Taking small enough $\e$, we have then $\vec{x}
\in A_T(>\e)$. This proves the assertion.
\end{proof}

Since $F(S)$ is compact there exists $\e(S)>0$ such that $F(S)
\subseteq \bigcup_{T \subseteq S}A_{T}(>\e(S))$. Let
$\delta=\frac{1}{2}\min\{\e(S)\mid S\subseteq [k+1]\}$ and for $T
\subseteq [k+1]$ define $B_T=A_T(\ge \delta)$. Then $F(S) \subseteq
\bigcup_{T \subseteq S}B_{T}$ for all $S \subseteq [k+1]$.

\begin{assertion}
The sets $B_T$ are closed.
\end{assertion}
\begin{proof}
 Let $\vec{x}$ be a limit point of the sequence  $\vec{x}^n \in
 B_T$. For every $n$ let $h^n$ be an edge witnessing the fact that $\vec{x}^n \in B_T$.
Since $H$ is finite, there is an edge $h \in H$ such that $h^n=h$
for infinitely many values of $n$. Then $h \subseteq L(T,\vec{x})$,
and its distance from the boundary of $L(T,\vec{x})$ is at least
$\delta$, meaning that $\vec{x} \in B_T$.
\end{proof}

We have shown that the sets $B_T$ satisfy the conditions of the KKMS
theorem, so by this theorem there exists a balanced collection $\T$
of discrete $d$-intervals in $[k+1]$ such that $\bigcap_{T\in \T}
B_{T} \neq \emptyset$. Then $\tau^*_{\ell}(\T) \le  \alpha
d\nu_{\ell}(\T)$, and by Lemma \ref{discreteweaker} there is a
matching $M \in \T$ with $\sum_{m \in M}|m| \ge \frac{k+1}{\alpha
d}$. Therefore,
\begin{equation}\label{eq:sumM}
\sum_{m \in M}(|m| - c(M,\vec{x}) +1) \ge \sum_{m \in M}\frac{|m|}{c(M,\vec{x})} \ge \frac{1}{d}\sum_{m \in M}|m| \ge \frac{k+1}{\alpha d^2}.
\end{equation}

Choose a point $\vec{x} \in \bigcap_{T \in \T} B_{T}$. For every $T \in \T$, the fact that $\vec{x} \in B_{T}$ means that there exists $h(T) \in H$ such that $h(T) \subset int(L(T,\vec{x}))$ and $w(h(T))\ge np(int(L(T,\vec{x})),\vec{x}) +1 \ge |T| - c(T,\vec{x}) +1$.
Therefore by (\ref{eq:sumM}) the proof of the theorem will be complete
if we show that for disjoint $T_1, T_2 \in \T$  the edges $h(T_1)$ and $h(T_2)$ are disjoint. If they do meet, then
any point $z \in h(T_1) \cap h(T_2)$ lies in the interior of a connected component
of $L(T_1,\vec{x})$ and in the interior of a connected component
of $L(T_2,\vec{x})$. There exists then $t$  such that $x_t \neq 0$, and  $z \in [p_{\vec{x}}({t-1}),p_{\vec{x}}(t)]$. Then  $t \in T_1 \cap T_2$, a contradiction. This concludes the proof of the theorem.


\section{A proof of Theorem~\ref{t:kaiser} using KKMS}
\label{sec.tardos-kaiser}

In this section we use the KKMS theorem to give a short proof of Theorem~\ref{t:kaiser}. In~\cite{kaiser}, a certain version of
Borsuk's theorem was used. The KKMS theorem and its generalizations
get to the point more directly.

We shall also use (as is done in \cite{kaiser}) a theorem of
F\"uredi \cite{furedi}:

\begin{theorem}\label{furedi}
If all edges of a hypergraph $H$ are of size at most $d$, then
$\nu(H) \ge \frac{\nu^*(H)}{d-1+\frac{1}{d}}$. If, in addition,  $d>2$ and $H$
does not contain a copy of the $d$-uniform projective plane then
$\nu(H) \ge \frac{\nu^*(H)}{d-1}$.
\end{theorem}

The first half of the following theorem just restates the non-separated
part of Theorem~\ref{t:kaiser}:
\begin{theorem}\label{thm:kaiser}\cite{kaiser} If $H$ is a hypergraph of  $d$-intervals then:
\begin{enumerate}

\item   $\tau(H) \le (d^2-d+1)\nu(H)$, and

\item  $\tau(H) \le d\nu^*(H)$.

\end{enumerate}

\end{theorem}

\begin{proof}
Like before, we assume that all edges of $H$ are contained in $(0,1)$. Our aim is to show that if $\tau(H)>k$ then $\nu(H) > \frac{k}{d^2-d+1}$. Every point
$\vec{x}=(x_1, x_2, \ldots ,x_{k+1})$ in $P = \Delta_{k}$
corresponds to a distribution of $k$ points $p_{\vec{x}}(m)=\sum_{i
\le m}x_i$, $1 \le m \le k$ on $L$. Set $p_{\vec{x}}(0) = 0$ and
$p_{\vec{x}}(k+1) = 1$.

For a subset $I$ of $[k+1]$ let $B_I$ consist of all vectors $\vec{x} \in
\Delta_k$ for which there exists an edge $h \in H$ satisfying: (a)
$h$ does not contain any point $p_{\vec{x}}(m)$, and (b) for each $i
\in I$ there exists at least one $j\le d$ such that $h^j \subseteq
(p_{\vec{x}}({i-1}),p_{\vec{x}}({i}))$.  Note that  if $B_I \neq
\emptyset$ then $|I| \le d$.

Clearly, the sets $B_{I}$ are open. By the assumption that $\tau
>k$, for every $\vec{x} \in P$ the points $p_{\vec{x}}(m)$, $1\le m
\le k$, do not cover $H$, meaning that there exists $h\in H$ not
containing any $p_{\vec{x}}(m)$. This, in turn, means that $\vec{x}
\in B_{I}$ for some $I \subseteq [k+1]$. We have thus shown that $P
=\bigcup B_{I}$.

Let $F=F(J)$ be a face of $\Delta_k$.  If $\vec{x}\in F(J)$ then
$(p_{\vec{x}}({i-1}),p_{\vec{x}}({i}))=\emptyset$ for $i \notin J$,
and hence it it impossible to have $h^j \subseteq
(p_{\vec{x}}({i-1}),p_{\vec{x}}({i}))$. Thus  $\vec{x} \in B_{I}$
for some $I \subseteq J$. This proves that $F \subseteq \bigcup_{I
\subseteq J}B_{I}$.

By Theorem \ref{kkms}
 there exists a balanced set  $\I$ of subsets of $[k+1]$, satifying:
\begin{enumerate}
\item
$\bigcap_{I \in \I} B_{I}\neq \emptyset$, and

\item
$|I|\le d$ for all $I \in \I$.

\end{enumerate}

Since $\I$ is balanced, (2) implies that $\nu^*(\I) \ge
\frac{k+1}{d}$. By Theorem \ref{furedi}, $\nu(\I) \ge
\frac{\nu^*(\I)}{d-1+\frac{1}{d}}\ge \frac{k+1}{d(d-1+\frac{1}{d})}
> \frac{k}{d^2-d+1}$.

Let $M$ be a matching in $\I$ of size at least $m=\lceil \frac{k}{d^2-d+1}\rceil$.

Let $\vec{x}$ be a point in $\bigcap_{I\in \I} B_{I}$. For every $I
\in \I$ let $h(I)$ be the edge of $H$ witnessing the fact that
$\vec{x} \in B_I$. Then the edges $h(I),~I \in M$ form a matching of
size $m$ in $H$, proving the lower bound on $\nu$.

To prove (2), let $f: \I \to \mathbb{R}^+$ be the fractional
matching of size at least $\frac{k+1}{d}$ whose existence we have
just proved. Then the function $\tilde{f} : H \to \mathbb{R}^+$
defined by $\tilde{f}(h)=f(I)$ if $h=h(I)$ and $\tilde{f}(h)=0$
otherwise, is a fractional matching of size at least $\frac{k+1}{d}$
in $H$.
\end{proof}

For the convenience of the reader, we restate also the  separated case of
Theorem~\ref{t:kaiser}:
\begin{theorem}\label{dintkaiser}
 In a hypergraph  $H$ of separated $d$-intervals,  $\tau(H) \le (d^2-d)\nu(H)$.
\end{theorem}

To prove this we use the following extension of the KKMS theorem
proved by Komiya \cite{komiya}:

\begin{theorem} \label{komiya}
Let $P$ be a polytope, and let a point $q(F)\in F$ be chosen for
every face $F$ of $P$. Let also $B_F$ be a an assignment of an open subset of $P$ to every face $F$, satisfying the condition
that every face $G$ is contained in $\bigcup_{F \subseteq G} B_F$.
Then there exists a collection $\F$ of faces  such that $q(P) \in
conv\{q(F) \mid F \in \F\}$ and $\bigcap_{F\in \F} B_{F}\neq
\emptyset$.
\end{theorem}

The theorem is true also if all $B_F$ are closed. The KKMS theorem is the case in which $P=\Delta_k$ and each $q(F)$ is the center of the face $F$.\\ \ \\
{\em Proof of Theorem \ref{dintkaiser}}. It clearly suffices to
prove that if $H$ is a hypergraph of $d$-intervals satisfying
$\tau(H)>kd$ then $\nu(H) \ge \frac{k+1}{d-1}$.  We may
assume that the ground set of $H$ is the $d$-fold product $(0,1)
\times \ldots \times (0,1)$. We apply Komiya's theorem to
$P=\Delta_k \times \Delta_k \times \ldots \times \Delta_k$, the
$d$-fold product of the $k$-dimensional simplex $\Delta_k$ by
itself. A point in $P$ has the form:
$$\vec{x}=((x^1_1, x^1_2, \ldots ,x^1_{k+1}), (x^2_1, x^2_2,
\ldots ,x^2_{k+1}), \ldots ,(x^d_1, x^d_2, \ldots ,x^d_{k+1}))$$
where $x^j_i \ge 0$ and $\sum_{i=1}^{k+1} x^j_i =1$ for every $j$.  For $\vec{x} \in P$ let
$p_{\vec{x}}(i,j)=\sum_{a \le i}x_a^j$.

Let $V$ be the set of all pairs $\{(i,j) \mid  1 \le i \le k+1,~ 1
\le j \le d\}$. A vertex $\vec{t}$ of $P$ corresponds to a $d$-tuple
$(i^1,1),(i^2,2), \ldots ,(i^d,d)$ of vertices in $V$, by the rule
$t^a_b =1$ if $b=i^a$ and $t^a_b=0$ otherwise (here we are using
with respect to $\vec{t}$ the same notation we used for points in
$P$ denoted by $\vec{x}$, namely $t^a_b$ is the $b$ coordinate of
$\vec{t}$ in the $a$-th copy of $(0,1)$). To any such vertex we can
assign an edge $e_{\vec{t}}=\{(i^1,1),(i^2,2), \ldots ,(i^d,d)\}$ in
the complete $d$-partite hypergraph with vertex set $V$ and sides
$V^j=\{(i,j) \mid 1 \le i \le k+1\}$. Let $B_{\vec{t}}$ be the set
of all points $\vec{x} \in P$ for which there exists $h \in H$
satisfying $h^j \subseteq (p_{\vec{x}}(i^j-1,j),
p_{\vec{x}}(i^j,j))$ for all $j$. Let also $q(\vec{t})=\vec{t}$ (the
only possible choice). For all other faces $F$ of $P$ let
$B_F=\emptyset$. Let $q(P)$  be the uniformly all $\frac{1}{k+1}$
vector. The  points $q(F)$ for all other faces $F$ do not come into
play, so we do not define them.

By our assumption, for no $\vec{x}\in P$ is the set of all points
$p_{\vec{x}}(i,j)$ ($1 \le i \le k$, $1 \le j \le d$) a cover for
$H$. Hence $\bigcup B_{\vec{t}}=P$, where the union is over all
vertices $\vec{t}$ of $P$. As in the previous proof, it is also easy
to see that for every face $F =conv(T)$ (where $T$ is a set of
vertices)
 $F \subseteq \bigcup \{B_{\vec{t}} \mid \vec{t} \in T\}$. By Theorem \ref{komiya}, there exists a set $Q$ of vertices, such that $q(P) \in conv(q(\vec{t}) \mid \vec{t} \in Q)$ and $\bigcap_{\vec{t} \in Q}B_{\vec{t}} \neq \emptyset$.

 Let  $E=\{e_{\vec{t}} \mid \vec{t} \in Q\}$. Then the hypergraph $D=(V,E)$ is $d$-partite,
 and the fact that $q(F) \in conv(q(\vec{t}) \mid \vec{t} \in Q)$
 means that $D$ is balanced. This in turn implies that $\nu^*(D)$ is $k+1$
 (the size of one side of $D$). Since $D$ is $d$-partite, by Theorem
 \ref{furedi}
$\nu(D) \ge \frac{\nu^*(D)}{d-1}\ge \frac{k+1}{d-1}$ (for $d=2$ we are using here K\"onig's theorem, rather than F\"uredi's theorem). Let
 $M$ be a matching in $D$ with $|M|\ge \frac{k+1}{d-1}$.
 Let $\vec{x}$ be a point in $\bigcap_{\vec{t} \in Q}B_{\vec{t}}$. By the definition of the sets $B_{\vec{t}}$ , for every edge $e=\{(i^1,1),(i^2,2), \ldots ,(i^d,d)\} \in M$ there exists an edge $h=h(e) \in H$ with   $h^j \subseteq (p_{\vec{x}}(i^j-1,j), p_{\vec{x}}(i^j,j))$. Clearly, the edges $h(e),~ e\in M$ are disjoint, proving that
 $\nu(H) \ge \frac{k+1}{d-1}$, as desired.
\hfill{$\square$}


\section{Bounding the ratio $\frac{\tau^*}{\nu}$ using topology}\label{toptau*nu}

As already mentioned, it is likely that in order to find the right
upper bound on $\frac{\tau^*}{\nu}$ a topological method will be
needed.  In this section we describe an approach that yields the
bound $4d$.

\begin{theorem}\label{main1}
If $H$ is a hypergraph of $d$-intervals then $\tau^*(H)
\le (4d-6+\frac{3}{d})\nu(H)$.
\end{theorem}

\begin{proof}
As before, we assume that all edges are contained
in $(0,1)$. We will show that if $\tau^*(H)>\alpha$ then
$\nu(H) > \frac{\alpha}{4d-6+\frac{3}{d}}$. Let $k=\lfloor \alpha d
\rfloor$. The assumption that $\tau^*(H)>\alpha$ implies that for
every point $\vec{x}=(x_1, x_2, \ldots ,x_{k+1})$ in $P =
\Delta_{k}$, putting weights $\frac{1}{d}$ on each of the points
$p_{\vec{x}}(m)=\sum_{i \le m}x_i$, $1 \le m \le k$, does not
constitute a fractional cover for $H$. Note that the points
$p_{\vec{x}}(m)$ are taken with multiplicity, namely if such a point
repeats $q$ times, it is given weight $\frac{q}{d}$. Therefore, for each $\vec{x} \in \Delta_k$ there exists $h \in H$ that contains
fewer than $d$ points $p_{\vec{x}}(m)$ (counted with multiplicity). This implies the following:

\begin{assertion}\label{uncoverededges}
For each $\vec{x} \in \Delta_k$ there exists $h \in H$ that meets at most $2d-1$ intervals of the form $[p_{\vec{x}}(m), p_{\vec{x}}(m+1)]$.
\end{assertion}
\begin{proof}
Let $h \in H$ be a $d$-interval that contains fewer than $d$ points
$p_{\vec{x}}(m)$. For each interval component $h^j$ of $h$, the
number of intervals $[p_{\vec{x}}(i),p_{\vec{x}}(i+1)]$ that $h^j$
meets is equal to the number of points $p_{\vec{x}}(i)$ it contains,
plus $1$. Since $h$ has at most $d$ interval components, it follows that the total number of intervals $[p_{\vec{x}}(i),p_{\vec{x}}(i+1)]$ that $h$ meets is at most $d-1+d=2d-1$.
\end{proof}

Given $\vec{x}\in P$, let $p_{\vec{x}}(0)=0$ and
$p_{\vec{x}}(k+1)=1$. For $T \subseteq [k+1]$ define $L(T,\vec{x})=
\bigcup_{t \in T}[p_{\vec{x}}({t-1}),p_{\vec{x}}(t)]$. We want now
to define sets $B_T$, towards an application of the KKMS theorem. Similarly to the case in the proof of Theorem \ref{reduction}, in order to make sure that the sets are closed we will use a compactness argument.

If $|T| \ge 2d$ we let $B_T$ be the empty set.

Suppose that $|T| \le 2d-1$. Let $A_{T}(>\e)$ (resp. $A_{T}(\ge
\e)$) be the set of those points $\vec{x}$ for which there exists an
edge  $h \in H$ contained in $int(L(T,\vec{x}))$ and satisfying
$dist(h,\partial(L(T,\vec{x})))
> \e$ (resp. $dist(h,\partial(L(T,\vec{x}))) \ge \e$.

\begin{assertion}
 $F(S) \subseteq \bigcup_{\e>0}
\bigcup_{T \subseteq S}A_{T}(>\e)$ for every subset $S$ of $[k+1]$.
\end{assertion}
\begin{proof}
Let $\vec{x}$ be a point in $F(S)$. By Assertion \ref{uncoverededges} there
exist $h \in H$ and $R \subseteq [k+1]$ such that $h \subseteq
int(L(R,\vec{x}))$ and $|R| \le 2d-1$. Let $T = R \cap S$. Since
$x_i=0$ for $i \notin S$ we have $int(L(R,\vec{x}))=int(L(T,\vec{x}))$. Taking
small enough $\e$, we have then $\vec{x} \in A_T(>\e)$. This proves
the assertion.
\end{proof}

As in the proof of Theorem \ref{reduction}, define $B_T=A_T(\ge \delta)$ for small enough $\delta$. Then $F(S) \subseteq
\bigcup_{T \subseteq S}B_{T}$ for all $S \subseteq [k+1]$.

The sets $B_T$ are closed and satisfy the conditions of the KKMS
theorem. Hence there exists a balanced collection $\T$
of subsets of $[k+1]$ such that $\bigcap_{T\in \T} B_{T} \neq
\emptyset$. Let $f:~\T \to \mathbb{R}^+$ be a perfect fractional
matching.
Since $|T| \le 2d-1$ for all $T \in \T$, it follows that $\nu^*(\T)
\ge \frac{|V(\T)|}{2d-1} = \frac{k+1}{2d-1}$. Therefore by Theorem
\ref{furedi} $\nu(\T) \ge \frac{k+1}{(2d-1)(2d-2 + \frac{1}{2d-1})}
> \frac{\alpha}{4d-6+\frac{3}{d}}$, namely there exists a matching
$M$ of size larger than $\frac{\alpha}{4d-6+\frac{3}{d}}$ in $\T$.

Choose a point $\vec{x} \in \bigcap_{T \in \T} B_{T}$. For every $T
\in \T$ let $h(T)$ be an edge of $H$ witnessing
$\vec{x} \in B_T$, a fact entailing $h(T) \subseteq
int(L(T,\vec{x}))$. Using the same argument as in the proof of Theorem \ref{reduction} we have that
for disjoint $T_1, T_2 \in \T$  the edges $h(T_1)$ and $h(T_2)$ are disjoint, which completes the proof of the theorem.

\end{proof}

\section*{Acknowledgment}
\label{sec:acknowledgment}

We thank Andr\'{a}s Gy\'{a}rf\'{a}s for kindly drawing our attention
to the example mentioned after Conjecture~\ref{conj:tau*nu}.

\end{document}